\documentclass[10pt,portrait,reqno,11pt]{amsart}%
\usepackage{amssymb,amsthm,amsmath, amsfonts, float}
\usepackage{graphicx}
\usepackage{multirow}
\usepackage{subfig}
\usepackage{amsmath}
\usepackage{amsfonts}
\usepackage{lscape}
\usepackage{amssymb}%
\providecommand{\U}[1]{\protect\rule{.1in}{.1in}}
\newtheorem{theorem}{Theorem}
\theoremstyle{plain}

\newtheorem{corollary}{Corollary}

\numberwithin{equation}{section}
\begin{document}
\title[Parallel Transport Frame Approach to Soliton Surf. associated with the B-DR
Eq. in $E^{4}$]{A Parallel Transport Frame Field Approach to Soliton Surfaces associated with
the Betchov-Da Rios Equation in Four Space}
\author[M. Alt\i n, A. Kazan]{\bfseries Mustafa Alt\i n$^{1\ast}$ and Ahmet Kazan$^{2}$}
\address{$^{1}$Department of Mathematics, Faculty of Arts and Sciences, Bing\"{o}l
University, Bing\"{o}l, Turkey \\
\newline
$^{2}$Department of Computer Technologies, Do\u{g}an\c{s}ehir Vahap
K\"{u}\c{c}\"{u}k Vocational School, Malatya Turgut \"{O}zal University,
Malatya, Turkey \\
\newline
$^{\ast}$Corresponding author: \textit{maltin@bingol.edu.tr}}

\begin{abstract}
In the present paper, the geometric properties of a soliton surface
$\Psi=\Psi(s,t)$ associated with the Betchov-Da Rios (B-DR) equation
using the parallel transport frame field in four-dimensional
Euclidean space are examined. We obtain the derivative formulas for
the parallel transport frame field of a unit-speed $s$-parameter
curve $\Psi=\Psi(s,t),$ for all $t$. We obtain the soliton surface's
two basic geometric invariants, $k$ and $h,$ and some other
important invariants such as Gaussian curvature, mean curvature
vector and Gaussian torsion. With the aid of these, a set of
theorems describing the conditions in which the soliton surface is
flat, minimal, semi-umbilic or Wintgen ideal (superconformal) are
proved using these surface invariants. Also, we give a theorem which
characterizes the curvature ellipse of the B-DR soliton surface with
respect to the parallel transport frame field in $E^{4}$. Finally,
we construct an example of a B-DR soliton surface, obtain its
geometric invariants and showed its projections into
three-dimensional space to demonstrate our theoretical
understanding.

\end{abstract}
\maketitle


\section{General Information and Basic Concepts}

The study of integrable curve dynamics is essential to comprehending a number
of physical phenomena, especially vortex motion and fluid dynamics. In this
regard, the vortex filament equation (VFE), which details the self-induced
motion of a vortex filament in an ideal fluid, is among the most prominent
instances. The VFE, also known as the smoke ring equation or the localized
induction equation (LIE), was first proposed by L.S. Da Rios in the early
1900s and models the evolution of space curves in $R^{3}$. In an
incompressible, inviscid three-dimensional fluid, it provides a mathematical
framework for examining the motion of one-dimensional vortex filaments.

If $\Psi=\Psi(s,t)$ is the position vector of a vortex filament,
then the governing equation takes the form
$\Psi_{t}=\Psi_{s}\times\Psi_{ss},$ where $s$ stands for the
arc-length parameter and $t$ denotes time. Also, with the aid of the
Frenet-Serret frame of a space curve $\gamma(s,t)$, it can be
rewritten as $\gamma_{t}=\mathcal{T}\times\kappa N=\kappa B$. Here,
$\mathcal{T},N$ and $B$ are the tangent, normal, and binormal
vectors, respectively and $\kappa$ is the curvature function of the
curve. It is assumed that the vortex filament in this formulation is
smooth and free of self-intersections. Da Rios laid the foundations
for contemporary vortex filament models with his important
exploration of the localized induction approximation (LIA), which
explains the velocity produced by a vortex line at an exterior
point.

The study of vortex filaments has expanded to greater dimensions as fluid
dynamics and geometric analysis have developed. Specifically, the Betchov-Da
Rios (B-DR) equation, which controls the evolution of space curves in $E^{4}$,
can be used to describe the motion of a thin vortex filament in an inviscid
fluid in a four-dimensional environment:%
\begin{equation}
\Psi_{t}=\Psi_{s}\times\Psi_{ss}\times\Psi_{sss}. \label{BDReq}%
\end{equation}

This equation offers a dynamical system framework for investigating the
geometric evolution of vortex filaments in higher-dimensional spaces; it is
also known as the generalized localized induction equation (LIE) in $E^{4}$.
Interest in comprehending the complex curvature and torsion characteristics of
space curves in four dimensions, as well as their applications in mathematical
physics and fluid mechanics, has increased since the invention of this
equation. For a deeper exploration of vortex filament dynamics and the B-DR
equation, we refer to studies such as \cite{Barros}, \cite{Betchov},
\cite{Rios}, \cite{Grbovic}, \cite{Hasimoto1}, \cite{Hasimoto2}, \cite{Melek},
and etc.

In order to thoroughly examine the geometric characteristics of curves and
(hyper)surfaces in three and higher dimensions, frame fields play impotant
role. One of the most popular tools for describing space curves is the Frenet
frame. This frame field has been widely extended to higher-dimensional spaces
from its traditional use in three-dimensional contexts, offering a more
comprehensive geometric understanding of surface and curve structure. On the
other hand, since the Frenet frame cannot be recognized at the locations where
the curvature is zero, differential geometers may require alternate frames.
Therefore, new alternative frames to the Frenet frame, such as the Bishop
frame (parallel transport frame), generalized Bishop frames, Darboux frame, or
extended Darboux frame have been defined, and the differential geometry of
curves and surfaces has begun to be taken into consideration in accordance
with these alternative frames (see \cite{Abdel}, \cite{Mahmut}, \cite{Altin},
\cite{Ates}, \cite{Bishop}, \cite{Sezgin}, \cite{Carmo}, \cite{Casorati},
\cite{Dogan}, \cite{Duldul}, \cite{Kazan}, \cite{Kazany}, \cite{Kazan2},
\cite{Kiziltug}, \cite{Kisi}, \cite{Nomoto}, and etc).

Now, let us recall the parallel transport frame along a curve in
four-dimensional Euclidean space $E^{4}$.

The authors of (\cite{gokcelik}) constructed an alternative frame known as a
parallel transport frame along the curve $\alpha$ in $E^{4}$ by using the
tangent vector $\mathcal{T}(s)$ and three relatively parallel vector fields
$\mathcal{P}_{1}(s)$, $\mathcal{P}_{2}(s)$ and $\mathcal{P}_{3}(s)$. Given a
Frenet frame $\left\{  \mathcal{T},N,B_{1},B_{2}\right\}  $ along a unit speed
curve $\alpha=\alpha(s):I\rightarrow E^{4}$ and the parallel transport frame
of the curve $\alpha$ denoted by $\left\{  \mathcal{T},\mathcal{P}%
_{1},\mathcal{P}_{2},\mathcal{P}_{3}\right\}  $, the relation can be written
as%
\[
\left\{
\begin{array}
[c]{l}%
\mathcal{T}(s)=\mathcal{T}(s),\\
N(s)=\cos\theta(s)\cos\psi(s)\mathcal{P}_{1}(s)+\left(  -\cos\phi(s)\sin
\psi(s)+\sin\phi(s)\sin\theta(s)\cos\psi(s)\right)  \mathcal{P}_{2}(s)\\
\text{ \ \ \ \ \ \ \ }+\left(  \sin\phi(s)\sin\psi(s)+\cos\phi(s)\sin
\theta(s)\cos\psi(s)\right)  \mathcal{P}_{3}(s),\\
B_{1}(s)=\cos\theta(s)\sin\psi(s)\mathcal{P}_{1}(s)+\left(  \cos\phi
(s)\cos\psi(s)+\sin\phi(s)\sin\theta(s)\sin\psi(s)\right)  \mathcal{P}%
_{2}(s)\\
\text{ \ \ \ \ \ \ \ }+\left(  -\sin\phi(s)\cos\psi(s)+\cos\phi(s)\sin
\theta(s)\sin\psi(s)\right)  \mathcal{P}_{3}(s),\\
B_{2}(s)=-\sin\theta(s)\mathcal{P}_{1}(s)+\sin\phi(s)\cos\theta(s)\mathcal{P}%
_{2}(s)+\cos\phi(s)\cos\theta(s)\mathcal{P}_{3}(s),
\end{array}
\right.  ~
\]
where $\theta(s),$ $\psi(s)$ and $\phi(s)$ are the Euler angles
\cite{Henderson}. Furthermore, the alternative parallel transport frame
equations are%
\begin{equation}
\left.
\begin{array}
[c]{l}%
\mathcal{T}_{u}=k_{1}\mathcal{P}_{1}+k_{2}\mathcal{P}_{2}+k_{3}\mathcal{P}%
_{3},\\
\left(  \mathfrak{\mathcal{P}}_{1}\right)  _{s}=-k_{1}\mathcal{T}\\
\left(  \mathfrak{\mathcal{P}}_{2}\right)  _{s}=-k_{2}\mathcal{T},\\
\left(  \mathfrak{\mathcal{P}}_{3}\right)  _{s}=-k_{3}\mathcal{T},
\end{array}
\right\}  \label{darboux}%
\end{equation}
where $k_{1}$, $k_{2}$ and $k_{3}$ are curvature functions according to
parallel transport frame of the curve and their expressions are as follows:%
\[
\left\{
\begin{array}
[c]{l}%
k_{1}(s)=\kappa(s)\cos\theta(s)\cos\psi(s),\\
k_{2}(s)=\kappa(s)\left(  -\cos\phi(s)\sin\psi(s)+\sin\phi(s)\sin\theta
(s)\cos\psi(s)\right)  ,\\
k_{3}(s)=\kappa(s)\left(  \sin\phi(s)\sin\psi(s)+\cos\phi(s)\sin\theta
(s)\cos\psi(s)\right)  .
\end{array}
\right.  ~
\]
The structure of this paper is as follows:

Section 2 includes a theorem which contains the derivative formulas of a
parallel transport frame field of a unit speed $s$-parameter curve $\Psi
=\Psi(s,t)$ for all $t$ associated with the B-DR equation in $E^{4}$. Section
3 gives some important geometric characterizations for the B-DR soliton
surface using the parallel transport frame field in $E^{4}$. In Section 4, we
give a classification about curvature ellipse of the B-DR soliton surface
according to the parallel transport frame field in $E^{4}$ and in Section 5,
we construct a soliton surface $\Psi(s,t)$ associated with the B-DR equation
and find the parallel transport frame field of the $s$-parameter curve
$\Psi(s,t)$ for all $t$ in $E^{4}$. Additionally, we obtain this soliton
surface's geometric invariants and curvatures.

\section{Exploring the B-DR Soliton Equation Using the Parallel Transport
Frame Field in $E^{4}$}

In this section, we will prove a theorem for a soliton surface $\Psi
=\Psi(s,t)$ associated with the B-DR equation and this theorem contains the
derivative formulas of a parallel transport frame field of a unit speed
$s$-parameter curve $\Psi=\Psi(s,t)$ for all $t$. We assume throughout this
study that, at least two of the curvature functions $k_{1}$, $k_{2}$ and
$k_{3}$ are non-zero.

\begin{theorem}
Let the $s$-parameter curve $\Psi=\Psi(s,t)$ be unit speed for all $t$ and
$\Psi=\Psi(s,t)$ be a solution of the B-DR equation according to the parallel
transport frame field in $E^{4}$. The derivative formulas of the parallel
transport frame field are%
\[
\left[
\begin{array}
[c]{c}%
\mathcal{T}_{s}\\
\left(  \mathfrak{\mathcal{P}}_{1}\right)  _{s}\\
\left(  \mathfrak{\mathcal{P}}_{2}\right)  _{s}\\
\left(  \mathfrak{\mathcal{P}}_{3}\right)  _{s}%
\end{array}
\right]  =\left[
\begin{array}
[c]{cccc}%
0 & k_{1} & k_{2} & k_{3}\\
-k_{1} & 0 & 0 & 0\\
-k_{1} & 0 & 0 & 0\\
-k_{1} & 0 & 0 & 0
\end{array}
\right]  \left[
\begin{array}
[c]{c}%
\mathcal{T}\\
\mathfrak{\mathcal{P}}_{1}\\
\mathfrak{\mathcal{P}}_{2}\\
\mathfrak{\mathcal{P}}_{3}%
\end{array}
\right]
\]
and%
\[
\left[
\begin{array}
[c]{c}%
\mathcal{T}_{t}\\
\left(  \mathfrak{\mathcal{P}}_{1}\right)  _{t}\\
\left(  \mathfrak{\mathcal{P}}_{2}\right)  _{t}\\
\left(  \mathfrak{\mathcal{P}}_{3}\right)  _{t}%
\end{array}
\right]  =\left[
\begin{array}
[c]{cccc}%
0 & \mathfrak{a}_{12} & \mathfrak{a}_{13} & \mathfrak{a}_{14}\\
-\mathfrak{a}_{12} & 0 & \mathfrak{a}_{23} & \mathfrak{a}_{24}\\
-\mathfrak{a}_{13} & -\mathfrak{a}_{23} & 0 & \mathfrak{a}_{34}\\
-\mathfrak{a}_{14} & -\mathfrak{a}_{24} & -\mathfrak{a}_{34} & 0
\end{array}
\right]  \left[
\begin{array}
[c]{c}%
\mathcal{T}\\
\mathfrak{\mathcal{P}}_{1}\\
\mathfrak{\mathcal{P}}_{2}\\
\mathfrak{\mathcal{P}}_{3}%
\end{array}
\right]  ,
\]
where%
\begin{align}
\mathfrak{a}_{12} &  =k_{3}\left(  k_{2}\right)  _{ss}-k_{2}\left(
k_{3}\right)  _{ss},\label{a12}\\
\mathfrak{a}_{13} &  =k_{1}\left(  k_{3}\right)  _{ss}-k_{3}\left(
k_{1}\right)  _{ss},\label{a13}\\
\mathfrak{a}_{14} &  =k_{2}\left(  k_{1}\right)  _{ss}-k_{1}\left(
k_{2}\right)  _{ss},\label{a14}\\
\mathfrak{a}_{23} &  =\int\left(  k_{3}\left(  k_{1}\left(  k_{1}\right)
_{ss}+k_{2}\left(  k_{2}\right)  _{ss}\right)  -\left(  k_{3}\right)
_{ss}\left(  k_{1}^{2}+k_{2}^{2}\right)  \right)  ds,\label{a23}\\
\mathfrak{a}_{24} &  =\int\left(  -k_{2}\left(  k_{1}\left(  k_{1}\right)
_{ss}+k_{3}\left(  k_{3}\right)  _{ss}\right)  +\left(  k_{2}\right)
_{ss}\left(  k_{1}^{2}+k_{3}^{2}\right)  \right)  ds,\label{a24}\\
\mathfrak{a}_{34} &  =\int\left(  k_{1}\left(  k_{2}\left(  k_{2}\right)
_{ss}+k_{3}\left(  k_{3}\right)  _{ss}\right)  -\left(  k_{1}\right)
_{ss}\left(  k_{2}^{2}+k_{3}^{2}\right)  \right)  ds.\label{a34}%
\end{align}

\end{theorem}

\begin{proof}
To begin, let $\Psi=\Psi(s,t)$ be a solution of the B-DR equation and the
$s$-parameter curve $\Psi=\Psi(s,t)$ be a unit speed curve for every $t$. The
derivative formulas of the parallel transport frame field according to "$s$"
are obtained from (\ref{darboux}) as%
\begin{equation}
\left.
\begin{array}
[c]{l}%
\mathcal{T}_{s}(s,t)=k_{1}(s,t)\mathcal{P}_{1}(s,t)+k_{2}(s,t)\mathcal{P}%
_{2}(s,t)+k_{3}(s,t)\mathcal{P}_{3}(s,t),\\
\left(  \mathfrak{\mathcal{P}}_{1}\right)  _{s}(s,t)=-k_{1}(s,t)\mathcal{T}%
(s,t),\\
\left(  \mathfrak{\mathcal{P}}_{2}\right)  _{s}(s,t)=-k_{1}(s,t)\mathcal{T}%
(s,t),\\
\left(  \mathfrak{\mathcal{P}}_{3}\right)  _{s}(s,t)=-k_{1}(s,t)\mathcal{T}%
(s,t).
\end{array}
\right\}  \label{Fiu}%
\end{equation}
Let's get the derivative formulas of the parallel transport frame field
according to "$t$". We must obtain the smooth functions $a_{ij},$
$i,j\in\{1,2,3,4\}$ of the equations%
\begin{equation}
\left.
\begin{array}
[c]{l}%
\mathcal{T}_{t}(s,t)=\mathfrak{a}_{11}(s,t)\mathcal{T}(s,t)+\sum_{i=2}%
^{4}\mathfrak{a}_{1i}(s,t)\mathcal{P}_{i-1}(s,t),\\
\left(  \mathfrak{\mathcal{P}}_{1}\right)  _{t}(s,t)=\mathfrak{a}%
_{21}(s,t)\mathcal{T}(s,t)+\sum_{i=2}^{4}\mathfrak{a}_{2i}(s,t)\mathcal{P}%
_{i-1}(s,t),\\
\left(  \mathfrak{\mathcal{P}}_{2}\right)  _{t}(s,t)=\mathfrak{a}%
_{31}(s,t)\mathcal{T}(s,t)+\sum_{i=2}^{4}\mathfrak{a}_{3i}(s,t)\mathcal{P}%
_{i-1}(s,t),\\
\left(  \mathfrak{\mathcal{P}}_{3}\right)  _{t}(s,t)=\mathfrak{a}%
_{41}(s,t)\mathcal{T}(s,t)+\sum_{i=2}^{4}\mathfrak{a}_{4i}(s,t)\mathcal{P}%
_{i-1}(s,t).
\end{array}
\right\}  \label{Fiveski}%
\end{equation}
From $\left\langle \mathcal{T},\mathcal{T}\right\rangle =\left\langle
\mathcal{P}_{i},\mathcal{P}_{i}\right\rangle =1$, $\left\langle \mathcal{T}%
,\mathcal{P}_{i}\right\rangle =0$ and $\left\langle \mathcal{P}_{i}%
,\mathcal{P}_{j}\right\rangle =0$, where $i,j=1,2,3$ and $j\neq i$, we have
$\mathfrak{a}_{ii}(s,t)=0$ and $\mathfrak{a}_{ij}(s,t)=-\mathfrak{a}%
_{ji}(s,t)$. Thus the equations (\ref{Fiveski}) become%
\begin{equation}
\left.
\begin{array}
[c]{l}%
\mathcal{T}_{t}(s,t)=\mathfrak{a}_{12}(s,t)\mathcal{P}_{1}(s,t)+\mathfrak{a}%
_{13}(s,t)\mathcal{P}_{2}(s,t)+\mathfrak{a}_{14}(s,t)\mathcal{P}_{3}(s,t),\\
\left(  \mathfrak{\mathcal{P}}_{1}\right)  _{t}(s,t)=-\mathfrak{a}%
_{12}(s,t)\mathcal{T}(s,t)+\mathfrak{a}_{23}(s,t)\mathcal{P}_{2}%
(s,t)+\mathfrak{a}_{24}(s,t)\mathcal{P}_{3}(s,t),\\
\left(  \mathfrak{\mathcal{P}}_{2}\right)  _{t}(s,t)=-\mathfrak{a}%
_{13}(s,t)\mathcal{T}(s,t)-\mathfrak{a}_{23}(s,t)\mathcal{P}_{1}%
(s,t)+\mathfrak{a}_{34}(s,t)\mathcal{P}_{3}(s,t),\\
\left(  \mathfrak{\mathcal{P}}_{3}\right)  _{t}(s,t)=-\mathfrak{a}%
_{14}(s,t)\mathcal{T}(s,t)-\mathfrak{a}_{24}(s,t)\mathcal{P}_{1}%
(s,t)-\mathfrak{a}_{34}(s,t)\mathcal{P}_{2}(s,t).
\end{array}
\right\}  \label{Fiv}%
\end{equation}
(Here we must note that, from now on for simplicity's sake, we will no longer
write $(s,t)$ in $\mathfrak{a}_{ij}(s,t)$, $\mathcal{T}(s,t)$, and so on.
Also, we will use the notation $\frac{\partial f}{\partial s}$ and $f_{s}$
interchangeably for a differentiable function $f$, and similarly with higher
order derivatives; i.e. $\frac{\partial^{2}f}{\partial s\partial t}$ is the
same as $f_{st}$, and so on.)

Let us find the functions $\mathfrak{a}_{12}$, $\mathfrak{a}_{13}$,
$\mathfrak{a}_{14}$, $\mathfrak{a}_{23}$, $\mathfrak{a}_{24}$ and
$\mathfrak{a}_{34}$. Using
\begin{equation}
\Psi_{s}=\mathcal{T} \label{omegau}%
\end{equation}
and (\ref{Fiu}), we get%
\begin{equation}
\Psi_{ss}=k_{1}\mathcal{P}_{1}+k_{2}\mathcal{P}_{2}+k_{3}\mathcal{P}_{3}
\label{omegauu}%
\end{equation}
and%
\begin{equation}
\Psi_{sss}=-Q^{2}\mathcal{T}+\left(  k_{1}\right)  _{s}\mathrm{\mathcal{P}%
}_{1}+\left(  k_{2}\right)  _{s}\mathcal{P}_{2}+\left(  k_{3}\right)
_{s}\mathfrak{\mathcal{P}}_{3}, \label{omegauuu}%
\end{equation}
where $Q=\sqrt{k_{1}^{2}+k_{2}^{2}+k_{3}^{2}}.~$From (\ref{omegau}%
)-(\ref{omegauuu}) and the B-DR equation (\ref{BDReq}), we have%
\begin{equation}
\Psi_{t}=P\mathrm{\mathcal{P}}_{1}+R\mathcal{P}_{2}+S\mathfrak{\mathcal{P}%
}_{3}, \label{omegav}%
\end{equation}
where, $P,$ $R,$ $S$ are not all zero and they are defined by%
\begin{equation}
\left.
\begin{array}
[c]{c}%
P=k_{3}\left(  k_{2}\right)  _{s}-k_{2}\left(  k_{3}\right)  _{s},\\
R=k_{1}\left(  k_{3}\right)  _{s}-k_{3}\left(  k_{1}\right)  _{s},\\
S=k_{2}\left(  k_{1}\right)  _{s}-k_{1}\left(  k_{2}\right)  _{s}.
\end{array}
\right\}  \label{prs}%
\end{equation}
On the other hand, from (\ref{Fiv}) and (\ref{omegau}) we get%
\begin{equation}
\Psi_{st}=\mathfrak{a}_{12}\mathfrak{\mathcal{P}}_{1}+\mathfrak{a}%
_{13}\mathfrak{\mathcal{P}}_{2}+\mathfrak{a}_{14}\mathfrak{\mathcal{P}}_{3}
\label{omegauv}%
\end{equation}
and from (\ref{Fiu}) and (\ref{omegav}) we have%
\begin{equation}
\Psi_{ts}=\left(  k_{3}\left(  k_{2}\right)  _{ss}-k_{2}\left(  k_{3}\right)
_{ss}\right)  \mathrm{\mathcal{P}}_{1}+\left(  k_{1}\left(  k_{3}\right)
_{ss}-k_{3}\left(  k_{1}\right)  _{ss}\right)  \mathcal{P}_{2}+\left(
k_{2}\left(  k_{1}\right)  _{ss}-k_{1}\left(  k_{2}\right)  _{ss}\right)
\mathfrak{\mathcal{P}}_{3}. \label{omegavu}%
\end{equation}
We know that we have the compatibility condition $f_{st}=f_{ts}$ for a $C^{2}%
$-function $f$. Thus from $\Psi_{st}=\Psi_{ts}$, (\ref{omegauv}) and
(\ref{omegavu}), we get (\ref{a12})-(\ref{a14}).

Now, let us give $\mathcal{T}_{st},$ $\mathcal{T}_{ts}$ and the equations
obtained by $\mathcal{T}_{st}=\mathcal{T}_{ts}$, and so on.

Using $\mathcal{T}_{st}=\mathcal{T}_{ts}$ in the equations%
\begin{align}
\mathcal{T}_{st}  &  =\left(  -\mathfrak{a}_{12}k_{1}-\mathfrak{a}_{13}%
k_{2}-\mathfrak{a}_{14}k_{3}\right)  \mathcal{T}+\left(  \left(  k_{1}\right)
_{t}-\mathfrak{a}_{23}k_{2}-\mathfrak{a}_{24}k_{3}\right)
\mathfrak{\mathcal{P}}_{1}\nonumber\\
&  +\left(  \left(  k_{2}\right)  _{t}+\mathfrak{a}_{23}k_{1}-\mathfrak{a}%
_{34}k_{3}\right)  \mathfrak{\mathcal{P}}_{2}+\left(  \left(  k_{3}\right)
_{t}+\mathfrak{a}_{24}k_{1}+\mathfrak{a}_{34}k_{2}\right)
\mathfrak{\mathcal{P}}_{3} \label{F1uv}%
\end{align}
and%
\begin{equation}
\mathcal{T}_{ts}=\left(  -\mathfrak{a}_{12}k_{1}-\mathfrak{a}_{13}%
k_{2}-\mathfrak{a}_{14}k_{3}\right)  \mathcal{T}+\left(  \left(
\mathfrak{a}_{12}\right)  _{s}\right)  \mathfrak{\mathcal{P}}_{1}+\left(
\left(  \mathfrak{a}_{13}\right)  _{s}\right)  \mathfrak{\mathcal{P}}%
_{2}+\left(  \left(  \mathfrak{a}_{14}\right)  _{s}\right)
\mathfrak{\mathcal{P}}_{3}, \label{F1vu}%
\end{equation}
we have%
\begin{align}
&  \left(  \mathfrak{a}_{12}\right)  _{s}=\left(  k_{1}\right)  _{t}%
-\mathfrak{a}_{23}k_{2}-\mathfrak{a}_{24}k_{3},\label{F1uv2}\\
&  \left(  \mathfrak{a}_{13}\right)  _{s}=\left(  k_{2}\right)  _{t}%
+\mathfrak{a}_{23}k_{1}-\mathfrak{a}_{34}k_{3},\label{F1uv3}\\
&  \left(  \mathfrak{a}_{14}\right)  _{s}=\left(  k_{3}\right)  _{t}%
+\mathfrak{a}_{24}k_{1}+\mathfrak{a}_{34}k_{2}. \label{F1uv4}%
\end{align}
If we use $\left(  \mathfrak{\mathcal{P}}_{1}\right)  _{st}=\left(
\mathfrak{\mathcal{P}}_{1}\right)  _{ts}$ in the equations%
\begin{equation}
\left(  \mathfrak{\mathcal{P}}_{1}\right)  _{st}=-\left(  \left(  \left(
k_{1}\right)  _{t}\right)  \mathcal{T}+\left(  \mathfrak{a}_{12}k_{1}\right)
\mathfrak{\mathcal{P}}_{1}+\left(  \mathfrak{a}_{13}k_{1}\right)
\mathfrak{\mathcal{P}}_{2}+\left(  \mathfrak{a}_{14}k_{1}\right)
\mathfrak{\mathcal{P}}_{3}\right)  \label{F2uv}%
\end{equation}
and%
\begin{equation}
\left(  \mathfrak{\mathcal{P}}_{1}\right)  _{ts} =\left(  -\mathfrak{a}%
_{23}k_{2}-\mathfrak{a}_{24}k_{3}-\left(  \mathfrak{a}_{12}\right)
_{s}\right)  \mathcal{T}+\left(  -\mathfrak{a}_{12}k_{1}\right)
\mathfrak{\mathcal{P}}_{1}+\left(  -\mathfrak{a}_{12}k_{2}+\left(  \mathfrak{a}_{23}\right)
_{s}\right)  \mathfrak{\mathcal{P}}_{2}+\left(  \left(  \mathfrak{a}%
_{24}\right)  _{s}-\mathfrak{a}_{12}k_{3}\right)  \mathfrak{\mathcal{P}}_{3},
\label{F2vu}%
\end{equation}
then we have equation (\ref{F1uv2}) and%
\begin{align}
\left(  \mathfrak{a}_{23}\right)  _{s}  &  =\mathfrak{a}_{12}k_{2}%
-\mathfrak{a}_{13}k_{1},\label{F2uv3}\\
\left(  \mathfrak{a}_{24}\right)  _{s}  &  =\mathfrak{a}_{12}k_{3}%
-\mathfrak{a}_{14}k_{1}. \label{F2uv4}%
\end{align}

We get (\ref{a23}) and (\ref{a24}) by using (\ref{a12})-(\ref{a14}) in
(\ref{F2uv3}) and (\ref{F2uv4}), respectively.

From $\left(  \mathfrak{\mathcal{P}}_{2}\right)  _{st}=\left(
\mathfrak{\mathcal{P}}_{2}\right)  _{ts}$,%
\begin{equation}
\left(  \mathfrak{\mathcal{P}}_{2}\right)  _{st}=-\left(  \left(  \left(
k_{2}\right)  _{t}\right)  \mathcal{T}+\left(  \mathfrak{a}_{12}k_{2}\right)
\mathfrak{\mathcal{P}}_{1}+\left(  \mathfrak{a}_{13}k_{2}\right)
\mathfrak{\mathcal{P}}_{2}+\left(  \mathfrak{a}_{14}k_{2}\right)
\mathfrak{\mathcal{P}}_{3}\right)  \label{F3uv}%
\end{equation}
and%
\begin{equation}
\left(  \mathfrak{\mathcal{P}}_{2}\right)  _{ts}=\left(  \mathfrak{a}%
_{23}k_{1}-\left(  \mathfrak{a}_{13}\right)  _{s}-\mathfrak{a}_{34}%
k_{3}\right)  \mathcal{T}+\left(  -\mathfrak{a}_{13}k_{1}-\left(
\mathfrak{a}_{23}\right)  _{s}\right)  \mathfrak{\mathcal{P}}_{1}+\left(
-\mathfrak{a}_{13}k_{2}\right)  \mathfrak{\mathcal{P}}_{2}+\left(  \left(
\mathfrak{a}_{34}\right)  _{s}-\mathfrak{a}_{13}k_{3}\right)
\mathfrak{\mathcal{P}}_{3}, \label{F3vu}%
\end{equation}
we get equation (\ref{F1uv3}), equation (\ref{F2uv3}) and%
\begin{equation}
\left(  \mathfrak{a}_{34}\right)  _{s}=\mathfrak{a}_{13}k_{3}-\mathfrak{a}%
_{14}k_{2}. \label{F3uv4}%
\end{equation}
We derive $\mathfrak{a}_{34}$ as (\ref{a34}) by using (\ref{a13}) and
(\ref{a14}) in the equation (\ref{F3uv4}). If we use%
\begin{equation}
\left(  \mathfrak{\mathcal{P}}_{3}\right)  _{st}=\left(  -\left(
k_{3}\right)  _{t}\right)  \mathcal{T}+\left(  -\mathfrak{a}_{12}k_{3}\right)
\mathfrak{\mathcal{P}}_{1}+\left(  -\mathfrak{a}_{13}k_{3}\right)
\mathfrak{\mathcal{P}}_{2}+\left(  -\mathfrak{a}_{14}k_{3}\right)
\mathfrak{\mathcal{P}}_{3} \label{F4uv}%
\end{equation}
and%
\begin{equation}
\left(  \mathfrak{\mathcal{P}}_{3}\right)  _{ts}=\left(  \mathfrak{a}%
_{24}k_{1}-\left(  \mathfrak{a}_{14}\right)  _{s}+\mathfrak{a}_{34}%
k_{2}\right)  \mathcal{T}+\left(  -\mathfrak{a}_{14}k_{1}-\left(
\mathfrak{a}_{24}\right)  _{s}\right)  \mathfrak{\mathcal{P}}_{1}+\left(
-\mathfrak{a}_{14}k_{3}-\left(  \mathfrak{a}_{34}\right)  _{s}\right)
\mathfrak{\mathcal{P}}_{2}+\left(  -\mathfrak{a}_{14}k_{3}\right)
\mathfrak{\mathcal{P}}_{3} \label{F4vu}%
\end{equation}
in $\left(  \mathfrak{\mathcal{P}}_{3}\right)  _{st}=\left(
\mathfrak{\mathcal{P}}_{3}\right)  _{ts}$, then we reach the equations
(\ref{F1uv4}), (\ref{F2uv4}) and (\ref{F3uv4}), again.
\end{proof}

Additionally, we can prove the following corollary:

\begin{corollary}
Let the $s$-parameter curve $\Psi=\Psi(s,t)$ be unit speed for all $t$ and
$\Psi=\Psi(s,t)$ be a solution of the B-DR equation according to the parallel
transport frame field in $E^{4}$. The following equations hold:%
\begin{align}
0 &  =-\left(  k_{1}\right)  _{t}-\left(  k_{2}\right)  _{s}\left(
k_{3}\right)  _{ss}+\left(  k_{3}\right)  _{s}\left(  k_{2}\right)
_{ss}\label{sari1}\\
&  +k_{2}\left(  \int\left(  -\left(  k_{3}\right)  _{ss}\left(  k_{1}%
^{2}+k_{2}^{2}\right)  +k_{3}\left(  k_{1}\left(  k_{1}\right)  _{ss}%
+k_{2}\left(  k_{2}\right)  _{ss}\right)  \right)  ds-\left(  k_{3}\right)
_{sss}\right)  \nonumber\\
&  +k_{3}\left(  \int\left(  -k_{2}\left(  k_{1}\left(  k_{1}\right)
_{ss}+k_{3}\left(  k_{3}\right)  _{ss}\right)  +\left(  k_{2}\right)
_{ss}\left(  k_{1}^{2}+k_{3}^{2}\right)  \right)  ds+\left(  k_{2}\right)
_{sss}\right)  ,\text{ \ }\nonumber
\end{align}%
\begin{align}
0 &  =-\left(  k_{2}\right)  _{t}-\left(  k_{3}\right)  _{s}\left(
k_{1}\right)  _{ss}+\left(  k_{1}\right)  _{s}\left(  k_{3}\right)
_{ss}\label{sari22}\\
&  +k_{3}\left(  \int\left(  -\left(  k_{1}\right)  _{ss}\left(  k_{2}%
^{2}+k_{3}^{2}\right)  +k_{1}\left(  k_{3}\left(  k_{3}\right)  _{ss}%
+k_{2}\left(  k_{2}\right)  _{ss}\right)  \right)  ds-\left(  k_{1}\right)
_{sss}\right)  \nonumber\\
\text{ } &  +k_{1}\left(  -\int\left(  -\left(  k_{3}\right)  _{ss}\left(
k_{1}^{2}+k_{2}^{2}\right)  +k_{3}\left(  k_{1}\left(  k_{1}\right)
_{ss}+k_{2}\left(  k_{2}\right)  _{s}\right)  \right)  ds+\left(
k_{3}\right)  _{sss}\right)  ,\nonumber
\end{align}%
\begin{align}
0 &  =-\left(  k_{3}\right)  _{t}+\left(  k_{2}\right)  _{s}\left(
k_{1}\right)  _{ss}-\left(  k_{1}\right)  _{s}\left(  k_{2}\right)
_{ss}\label{sari3}\\
&  +k_{2}\left(  -\int\left(  -\left(  k_{1}\right)  _{ss}\left(  k_{2}%
^{2}+k_{3}^{2}\right)  +k_{1}\left(  k_{3}\left(  k_{3}\right)  _{ss}%
+k_{2}\left(  k_{2}\right)  _{ss}\right)  \right)  ds+\left(  k_{1}\right)
_{sss}\right)  \nonumber\\
&  -k_{1}\left(  \int\left(  -k_{2}\left(  k_{1}\left(  k_{1}\right)
_{ss}+k_{3}\left(  k_{3}\right)  _{ss}\right)  +\left(  k_{2}\right)
_{ss}\left(  k_{1}^{2}+k_{3}^{2}\right)  \right)  ds+\left(  k_{2}\right)
_{sss}\right)  .\text{ \ }\nonumber
\end{align}

\end{corollary}

\begin{proof}
We obtain (\ref{sari1}) by utilizing (\ref{a23}) and (\ref{a24}) in the
equation (\ref{F1uv2}).

Additionally, we have (\ref{sari22}) by utilizing (\ref{a23}) and (\ref{a34})
in the equation (\ref{F1uv3}).

Finally, we get (\ref{sari3}) by utilizing (\ref{a24}) and (\ref{a34}) in
(\ref{F2uv4}).
\end{proof}

\section{A Geometric Analysis of the B-DR Soliton Surface Using the Parallel
Transport Frame Field in $E^{4}$}

This section contains two invariants $k$ and $h$ introduced in \cite{Ganchev}
of a two-dimensional B-DR soliton surface $S:\Psi=\Psi(s,t)$ according to the
parallel transport frame field in $E^{4}$ and additionally, some
characterizations for this surface by obtaining its Gaussian curvature, mean
curvature vector field and Gaussian torsion.

Firstly, the coefficients of the first fundamental form is obtained as%
\begin{equation}
\left.
\begin{array}
[c]{l}%
g_{11}=\left\langle \Psi_{s},\Psi_{s}\right\rangle =1,\\
g_{12}=g_{21}=\left\langle \Psi_{s},\Psi_{t}\right\rangle =0,\\
g_{22}=\left\langle \Psi_{t},\Psi_{t}\right\rangle =P^{2}+R^{2}+S^{2}%
\end{array}
\right\}  \label{gij}%
\end{equation}
From (\ref{gij}), let us set%
\begin{equation}
\mathcal{W=}\sqrt{g_{11}g_{22}-(g_{12})^{2}}=\sqrt{P^{2}+R^{2}+S^{2}}.
\label{W}%
\end{equation}

If $T_{p}(S)=span\{\Psi_{s}=T,\Psi_{t}=P\mathrm{\mathcal{P}}_{1}%
+R\mathcal{P}_{2}+S\mathfrak{\mathcal{P}}_{3}\}$ is the tangent space of the
B-DR soliton surface $S:\Psi=\Psi(s,t)$ according to the parallel transport
frame field in $E^{4}$, then the orthonormal normal frame fields $N_{1}$ and
$N_{2}$ of the normal space $N_{p}(S)=span\{N_{1},N_{2}\}$ can be obtained as
\begin{equation}
\left.
\begin{array}
[c]{l}%
N_{1}=\frac{k_{1}\mathfrak{\mathcal{P}}_{1}+k_{2}\mathfrak{\mathcal{P}}%
_{2}+k_{3}\mathfrak{\mathcal{P}}_{3}}{Q},\\
\\
N_{2}=\frac{\left(  k_{2}S-k_{3}R\right)  \mathfrak{\mathcal{P}}_{1}+\left(
k_{3}P-k_{1}S\right)  \mathfrak{\mathcal{P}}_{2}+\left(  k_{1}R-k_{2}P\right)
\mathfrak{\mathcal{P}}_{3}}{Q\mathcal{W}}.
\end{array}
\right\}  \label{N1N2}%
\end{equation}

Let $\Gamma_{ij}^{k}$ $(i,j,k=1,2)$ be the Christoffel's symbols and
$c_{ij}^{k}$ be functions on $S$. The orthonormal normal frame field
$\{N_{1},N_{2}\}$ of $S$ then has the typical derivative formulas shown below:%
\begin{equation}
\left.
\begin{array}
[c]{l}%
\Psi_{ss}=\Gamma_{11}^{1}\Psi_{s}+\Gamma_{11}^{2}\Psi_{t}+c_{11}^{1}%
N_{1}+c_{11}^{2}N_{2},\\
\Psi_{st}=\Gamma_{12}^{1}\Psi_{s}+\Gamma_{12}^{2}\Psi_{t}+c_{12}^{1}%
N_{1}+c_{12}^{2}N_{2},\\
\Psi_{tt}=\Gamma_{22}^{1}\Psi_{s}+\Gamma_{22}^{2}\Psi_{t}+c_{22}^{1}%
N_{1}+c_{22}^{2}N_{2}.
\end{array}
\right\}  \label{omegaij}%
\end{equation}
On the other hand, from (\ref{a13}), (\ref{a23}), (\ref{a34}), (\ref{Fiv}) and
(\ref{omegav}) we get%
\begin{equation}
\Psi_{tt}=\left(  -PP_{s}-RR_{s}-SS_{s}\right)  \mathcal{T}+\left(
P_{t}-\mathfrak{a}_{23}R-\mathfrak{a}_{24}S\right)  \mathfrak{\mathcal{P}}%
_{1}+\left(  R_{t}+\mathfrak{a}_{23}P-\mathfrak{a}_{34}S\right)
\mathfrak{\mathcal{P}}_{2}+\left(  S_{t}+\mathfrak{a}_{24}P+\mathfrak{a}%
_{34}R\right)  \mathfrak{\mathcal{P}}_{3}. \label{omegavvy}%
\end{equation}
So, from (\ref{omegauu}), (\ref{omegavu}) and (\ref{omegavvy}), we have%
\begin{equation}
\left.
\begin{array}
[c]{l}%
c_{11}^{1}=\left\langle \Psi_{ss},N_{1}\right\rangle =Q,\text{ \ }c_{11}%
^{2}=\left\langle \Psi_{ss},N_{2}\right\rangle =0,\\
\\
c_{12}^{1}=\left\langle \Psi_{st},N_{1}\right\rangle =0,\text{ \ }c_{12}%
^{2}=\left\langle \Psi_{st},N_{2}\right\rangle =\frac{QC}{\mathcal{W}},\\
\\
c_{22}^{1}=\left\langle \Psi_{tt},N_{1}\right\rangle =\frac{A}{Q},\text{
\ }c_{22}^{2}=\left\langle \Psi_{tt},N_{2}\right\rangle =\frac{B}%
{\mathcal{W}Q},
\end{array}
\right\}  \label{cijk}%
\end{equation}
where%
\[
\left.
\begin{array}
[c]{l}%
\mathcal{A}=\left(  Pk_{3}-Sk_{1}\right)  \mathfrak{a}_{24}+\left(
Pk_{2}-Rk_{1}\right)  \mathfrak{a}_{23}+\left(  Rk_{3}-Sk_{2}\right)
\mathfrak{a}_{34}-R\left(  k_{2}\right)  _{t}-S\left(  k_{3}\right)
_{t}-P\left(  k_{1}\right)  _{t},\\
\mathcal{B}=\mathcal{W}^{2}\left(  k_{1}\mathfrak{a}_{34}-k_{2}\mathfrak{a}%
_{24}+k_{3}\mathfrak{a}_{23}\right)  +S\left(  k_{2}P_{t}-k_{1}R_{t}\right)
+R\left(  k_{1}S_{t}-k_{3}P_{t}\right)  +P\left(  k_{3}R_{t}-k_{2}%
S_{t}\right)  ,\\
\mathcal{C}=\left(  k_{1}\right)  _{s}P_{s}+\left(  k_{2}\right)  _{s}%
R_{s}+\left(  k_{3}\right)  _{s}S_{s}.
\end{array}
\right\}
\]

Let we introduce the following functions:%
\begin{equation}
\Delta_{1}=\left\vert
\begin{array}
[c]{cc}%
c_{11}^{1} & c_{12}^{1}\\
c_{11}^{2} & c_{12}^{2}%
\end{array}
\right\vert =\frac{Q^{2}\mathcal{C}}{\mathcal{W}},\text{ }\Delta
_{2}=\left\vert
\begin{array}
[c]{cc}%
c_{11}^{1} & c_{22}^{1}\\
c_{11}^{2} & c_{22}^{2}%
\end{array}
\right\vert =\frac{\mathcal{B}}{\mathcal{W}},\text{ }\Delta_{3}=\left\vert
\begin{array}
[c]{cc}%
c_{12}^{1} & c_{22}^{1}\\
c_{12}^{2} & c_{22}^{2}%
\end{array}
\right\vert =\frac{-\mathcal{AC}}{\mathcal{W}}.\label{delta123}%
\end{equation}
Then, we find the coefficients of the second fundamental form as%
\begin{equation}
l_{11}=\frac{2\Delta_{1}}{\mathcal{W}}=\frac{2Q^{2}\mathcal{C}}{\mathcal{W}%
^{2}},\text{ }l_{12}=\frac{\Delta_{2}}{\mathcal{W}}=\frac{\mathcal{B}%
}{\mathcal{W}^{2}},\text{ }l_{22}=\frac{2\Delta_{3}}{\mathcal{W}}%
=\frac{-2\mathcal{AC}}{\mathcal{W}^{2}}.\label{lij}%
\end{equation}

Moreover, if we consider the linear map
\[
\gamma:T_{\mathcal{P}}S\longrightarrow T_{\mathcal{P}}S
\]
which satisfies the conditions%
\[
\left.
\begin{array}
[c]{c}%
\gamma(\Omega_{s})=\gamma_{1}^{1}\Omega_{s}+\gamma_{1}^{2}\Omega_{t},\\
\gamma(\Omega_{t})=\gamma_{2}^{1}\Omega_{s}+\gamma_{2}^{2}\Omega_{t},
\end{array}
\right\}  \text{ \ \ }\left(  \gamma=\left[
\begin{array}
[c]{cc}%
\gamma_{1}^{1} & \gamma_{1}^{2}\\
\gamma_{2}^{1} & \gamma_{2}^{2}%
\end{array}
\right]  \right)  ,
\]
then we obtain that%
\begin{equation}
\left.
\begin{array}
[c]{l}%
\gamma_{1}^{1}=\frac{g_{12}l_{12}-g_{22}l_{11}}{g_{11}g_{22}-(g_{12})^{2}%
}=-\frac{2Q^{2}\mathcal{C}}{\mathcal{W}^{2}},\text{ \ }\gamma_{1}^{2}%
=\frac{g_{12}l_{11}-g_{11}l_{12}}{g_{11}g_{22}-(g_{12})^{2}}=-\frac
{\mathcal{B}}{\mathcal{W}^{4}},\\
\\
\gamma_{2}^{1}=\frac{g_{12}l_{22}-g_{22}l_{12}}{g_{11}g_{22}-(g_{12})^{2}%
}=-\frac{\mathcal{B}}{\mathcal{W}^{2}},\text{ \ }\gamma_{2}^{2}=\frac
{g_{12}l_{12}-g_{11}l_{22}}{g_{11}g_{22}-(g_{12})^{2}}=\frac{2\mathcal{AC}%
}{\mathcal{W}^{4}}.
\end{array}
\right\}  \label{gammaij}%
\end{equation}
Therefore, we can give the following theorem:

\begin{theorem}
If $\Psi=\Psi(s,t)$ is a solution of the B-DR equation, then
\begin{equation}
k(s,t)=-\frac{\mathcal{B}^{2}+4Q^{2}\mathcal{AC}^{2}}{\mathcal{W}^{6}}
\label{kuv1}%
\end{equation}
and%
\begin{equation}
h(s,t)=\frac{\left(  Q^{2}\mathcal{W}^{2}-\mathcal{A}\right)  \mathcal{C}%
}{\mathcal{W}^{4}} \label{huv}%
\end{equation}
are the invariants of the soliton surface $S:\Psi=\Psi(s,t)$ according to the
ED$^{2}$-frame field in $E^{4}$.
\end{theorem}

\begin{proof}
From (\ref{gammaij}), $k(s,t)=\det(\gamma(s,t))$ and $h(s,t)=-\frac{tr(\gamma(s,t))}{2},$ we obtain the invariants as (\ref{kuv1}) and (\ref{huv}).
\end{proof}

Otherwise, we find the coefficients of the shape operator matrices of the
soliton surface $S$ from (\ref{gij}), (\ref{W}), (\ref{N1N2}) and (\ref{cijk})
as%
\begin{equation}
\left.
\begin{array}
[c]{l}%
h_{11}^{1}=\frac{c_{11}^{1}}{g_{11}}=Q,\text{ \ }h_{11}^{2}=\frac{c_{11}^{2}%
}{g_{11}}=0,\\
\\
h_{12}^{1}=\frac{1}{\mathcal{W}}\left(  c_{12}^{1}-\frac{g_{12}}{g_{11}}%
c_{11}^{1}\right)  =0,\text{ \ }h_{12}^{2}=\frac{1}{\mathcal{W}}\left(
c_{12}^{2}-\frac{g_{12}}{g_{11}}c_{11}^{2}\right)  =\frac{Q\mathcal{C}%
}{\mathcal{W}^{2}},\\
\\
h_{22}^{1}=\frac{1}{\mathcal{W}^{2}}\left(  g_{11}c_{22}^{1}-2g_{12}c_{12}%
^{1}+\frac{(g_{12})^{2}}{g_{11}}c_{11}^{1}\right)  =\frac{\mathcal{A}%
}{Q\mathcal{W}^{2}},\text{ \ }\\
\\
h_{22}^{2}=\frac{1}{\mathcal{W}^{2}}\left(  g_{11}c_{22}^{2}-2g_{12}c_{12}%
^{2}+\frac{(g_{12})^{2}}{g_{11}}c_{11}^{2}\right)  =\frac{\mathcal{B}%
}{Q\mathcal{W}^{3}}.
\end{array}
\right\}  \label{hij}%
\end{equation}
With the aid of (\ref{N1N2}) and (\ref{hij}), we find the shape operator
matrices of the soliton surface $S$ as%
\begin{equation}
A_{N_{1}}=\left[
\begin{array}
[c]{cc}%
h_{11}^{1} & h_{12}^{1}\\
h_{12}^{1} & h_{22}^{1}%
\end{array}
\right]  \text{ and }A_{N_{2}}=\left[
\begin{array}
[c]{cc}%
h_{11}^{2} & h_{12}^{2}\\
h_{12}^{2} & h_{22}^{2}%
\end{array}
\right]  . \label{AF12}%
\end{equation}
Now, we can obtain the Gaussian curvature, mean curvature vector field and
Gaussian torsion of the soliton surface $S$. Also, we can give some important
geometric characterizations such as minimal, flat, semi-umbilic and Wintgen
ideal soliton surfaces according to the parallel transport frame field in
$E^{4}$.\

Here, initially, let us prove the following theorem which contains the
Gaussian curvature of the soliton surface $S$.

\begin{theorem}
If $\Psi=\Psi(s,t)$ is a solution of the B-DR equation according to the
parallel transport frame field in $E^{4}$, then the Gaussian curvature of the
soliton surface $S:\Psi=\Psi(s,t)$ is%
\begin{equation}
K=\frac{\mathcal{AW}^{2}-Q^{2}\mathcal{C}^{2}}{\mathcal{W}^{4}}.
\label{Gaussian}%
\end{equation}

\end{theorem}

\begin{proof}
Using (\ref{hij}) and (\ref{AF12}), we obtain the Gaussian curvature of $S$
from%
\[
K=\det(A_{N_{1}})+\det(A_{N_{2}}).
\]

\end{proof}

From (\ref{Gaussian}) and the statement "The surface is flat if and only if
its Gaussian curvature is zero", we can give the following theorem without proof.

\begin{theorem}
Let $\Psi=\Psi(s,t)$ be a solution of the B-DR equation according to the
parallel transport frame field in $E^{4}$. The soliton surface $S:\Psi
=\Psi(s,t)$ is flat if and only if $\mathcal{AW}^{2}=Q^{2}\mathcal{C}^{2}$ holds.\
\end{theorem}

Here, let us prove the following theorem which contains the mean curvature
vector field of the soliton surface $S:\Psi=\Psi(s,t)$. With the aid of this
theorem, we can give a characterization about this surface's minimality.

\begin{theorem}
If $\Psi=\Psi(s,t)$ is a solution of the B-DR equation according to the
parallel transport frame field in $E^{4}$, then the mean curvature vector
field of the soliton surface $S:\Psi=\Psi(s,t)$ is%
\begin{equation}
\vec{H}=\frac{1}{2Q^{2}\mathcal{W}^{4}}\left(
\begin{array}
[c]{c}%
\left(  \left(  Sk_{2}-Rk_{3}\right)  \mathcal{B}+\mathcal{W}^{2}k_{1}\left(
Q^{2}\mathcal{W}^{2}+\mathcal{A}\right)  \right)  \mathfrak{\mathcal{P}}_{1}\\
+\left(  \left(  -Sk_{1}+Pk_{3}\right)  \mathcal{B}+\mathcal{W}^{2}%
k_{2}\left(  Q^{2}\mathcal{W}^{2}+\mathcal{A}\right)  \right)
\mathfrak{\mathcal{P}}_{2}\\
+\left(  \left(  Rk_{1}-Pk_{2}\right)  \mathcal{B}+\mathcal{W}^{2}k_{3}\left(
Q^{2}\mathcal{W}^{2}+\mathcal{A}\right)  \right)  \mathfrak{\mathcal{P}}_{3}%
\end{array}
\right)  . \label{mean}%
\end{equation}

\end{theorem}

\begin{proof}
From (\ref{hij}), (\ref{AF12}) we have%
\begin{equation}
\vec{H}  =\frac{1}{2}\left(  tr(A_{N_{1}})N_{1}+tr(A_{N_{2}})N_{2}\right)=\frac{1}{2}\left(  \frac{Q^{2}\mathcal{W}^{2}+\mathcal{A}}{\mathcal{W}%
^{2}Q}N_{1}+\frac{\mathcal{B}}{\mathcal{W}^{3}Q}N_{2}\right)  . \label{my}%
\end{equation}
Using (\ref{N1N2}) in (\ref{my}), we obtain the mean curvature vector field of
$S$ as (\ref{mean}).
\end{proof}

So,

\begin{theorem}
\label{teominimal}Let $\Psi=\Psi(s,t)$ be a solution of the B-DR equation
according to the parallel transport frame field in $E^{4}$. The soliton
surface $S:\Psi=\Psi(s,t)$ is minimal if and only if the equations
$\mathcal{B}=0$ and $Q^{2}\mathcal{W}^{2}+\mathcal{A}=0$ hold.
\end{theorem}

\begin{proof}
From (\ref{mean}), the soliton surface $S$ is minimal if and only if%
\begin{equation}
\left.
\begin{array}
[c]{c}%
\left(  Sk_{2}-Rk_{3}\right)  \mathcal{B}+\mathcal{W}^{2}k_{1}\left(
Q^{2}\mathcal{W}^{2}+\mathcal{A}\right)  =0,\\
\left(  Pk_{3}-Sk_{1}\right)  \mathcal{B}+\mathcal{W}^{2}k_{2}\left(
Q^{2}\mathcal{W}^{2}+\mathcal{A}\right)  =0,\\
\left(  Rk_{1}-Pk_{2}\right)  \mathcal{B}+\mathcal{W}^{2}k_{3}\left(
Q^{2}\mathcal{W}^{2}+\mathcal{A}\right)  =0.
\end{array}
\right\}  \label{m1}%
\end{equation}
If the equations in (\ref{m1}) are considered binary, the following equations
are obtained by using (\ref{prs})%
\begin{equation}
\left.
\begin{array}
[c]{c}%
\mathcal{B}\left(  k_{1}^{2}+k_{2}^{2}+k_{3}^{2}\right)  \left(  k_{1}\left(
k_{2}\right)  _{s}-k_{2}\left(  k_{1}\right)  _{s}\right)  =0,\\
\mathcal{B}\left(  k_{1}^{2}+k_{2}^{2}+k_{3}^{2}\right)  \left(  k_{1}\left(
k_{3}\right)  _{s}-k_{3}\left(  k_{1}\right)  _{s}\right)  =0,\\
\mathcal{B}\left(  k_{1}^{2}+k_{2}^{2}+k_{3}^{2}\right)  \left(  k_{3}\left(
k_{2}\right)  _{s}-k_{2}\left(  k_{3}\right)  _{s}\right)  =0.
\end{array}
\right\}  \label{m3}%
\end{equation}
Thus from the equations in (\ref{m3}), we reach that $\mathcal{B}=0.$ Using
$\mathcal{B}=0$ in (\ref{m1}), we have%
\[
k_{1}\left(  Q^{2}\mathcal{W}^{2}+\mathcal{A}\right)  =k_{2}\left(
Q^{2}\mathcal{W}^{2}+\mathcal{A}\right)  =k_{3}\left(  Q^{2}\mathcal{W}%
^{2}+\mathcal{A}\right)  =0
\]
and so, it must be $Q^{2}\mathcal{W}^{2}+\mathcal{A}=0$ and this completes the proof.
\end{proof}

On the other hand from (\ref{kuv1}) and (\ref{huv}), we get%
\begin{equation}
h^{2}-k=\frac{\mathcal{B}^{2}\mathcal{W}^{2}+\mathcal{C}^{2}\left(
\mathcal{A+}Q^{2}\mathcal{W}^{2}\right)  ^{2}}{\mathcal{W}^{8}}.
\label{hkareeksik}%
\end{equation}

We know that, if $S$ is a surface in $E^{4}$ without flat points, then $S$ is
minimal if and only if $h^{2}-k=0$ \cite{Ganchev}. So, under the minimality
conditions which have been stated in Theorem \ref{teominimal}, one can see
that $h^{2}-k=0,$ too.

Now, we will obtain the Gaussian torsion of the soliton surface $S$ and give a
theorem for semi-umbilic soliton surface.

The Gaussian torsion (also called the normal curvature function) of a surface
$M\subset E^{4}$ given by a regular patch $\Psi(s,t)$ is (\cite{Aminov},
\cite{Desmet}, \cite{Guadalup}, \cite{Gutierrez})
\begin{equation}
K_{N}=\frac{g_{11}\left(  c_{12}^{1}c_{22}^{2}-c_{12}^{2}c_{22}^{1}\right)
-g_{12}\left(  c_{11}^{1}c_{22}^{2}-c_{11}^{2}c_{22}^{1}\right)
+g_{22}\left(  c_{11}^{1}c_{12}^{2}-c_{11}^{2}c_{12}^{1}\right)  }%
{\mathcal{W}^{3}}. \label{KN}%
\end{equation}
Thus, from (\ref{gij}), (\ref{W}), (\ref{cijk}) and (\ref{KN}), we have

\begin{theorem}
If $\Psi=\Psi(s,t)$ is a solution of the B-DR equation according to the
parallel transport frame field in $E^{4}$, then the Gaussian torsion of the
soliton surface $S:\Psi=\Psi(s,t)$ is%
\begin{equation}
K_{N}=\frac{\left(  Q^{2}\mathcal{W}^{2}-\mathcal{A}\right)  \mathcal{C}%
}{\mathcal{W}^{4}}. \label{KNy}%
\end{equation}

\end{theorem}

We know that, a point $p\in M$ is semi-umbilic if and only if $K_{N}(p)=0$ and
a surface $M$ immersed in $E^{4}$ is called semi-umbilical provided all its
points are semi-umbilic \cite{Gutierrez}. Therefore, from (\ref{a34}) and
(\ref{KNy}), we can give the following theorem:

\begin{theorem}
Let $\Psi=\Psi(s,t)$ be a solution of the B-DR equation according to the
parallel transport frame field in $E^{4}$. The soliton surface $S:\Psi
=\Psi(s,t)$ is semi-umbilic if and only if $\left(  Q\mathcal{W}\right)
^{2}=\mathcal{A}$ or $\mathcal{C}=0$ holds.
\end{theorem}

Now, we prove a theorem that characterizes the Wintgen ideal (superconformal)
B-DR soliton surface according to the parallel transport frame field in
$E^{4}$.

In 1979, Wintgen has proved the important inequality%
\[
K+\left\vert K_{N}\right\vert \leq\left\Vert \vec{H}\right\Vert ^{2}%
\]
for Gaussian curvature $K$, mean curvature vector field $\vec{H}$ and Gaussian
torsion $K_{N}$ of a surface in $E^{4}$ \cite{Wintgen}. Furthermore, the
equality, i.e.
\begin{equation}
K+\left\vert K_{N}\right\vert =\left\Vert \vec{H}\right\Vert ^{2}
\label{wintgen}%
\end{equation}
holds if and only if the curvature ellipse is a circle. With the aid of the
equation (\ref{wintgen}), Wintgen ideality or superconformality of a surface
in $E^{4}$ can be defined as "A surface in $E^{4}$ is called a Wintgen ideal
(superconformal) surface if it satisfies the equation (\ref{wintgen})".

Thus, we can state the following theorem which states the necessary conditions
for a B-DR soliton surface to be Wintgen ideal according to the parallel
transport frame field in $E^{4}$:

\begin{theorem}
Let $\Psi=\Psi(s,t)$ be a solution of the B-DR equation according to the
parallel transport frame field in $E^{4}$. The soliton surface $S:\Psi
=\Psi(s,t)$ is Wintgen ideal (superconformal) if and only if%
\begin{align*}
0  & =4\mathcal{W}^{4}Q^{4}\left(  \mathcal{A}+\mathcal{C}Q^{2}\right)
\left(  \mathcal{C}-\mathcal{W}^{2}\right)  +\left(  \mathcal{B}%
Rk_{1}+\mathcal{AW}^{2}k_{3}+\mathcal{W}^{4}Q^{2}k_{3}-\mathcal{B}%
Pk_{2}\right)  ^{2}\\
& +\left(  \mathcal{AW}^{2}k_{1}+\mathcal{W}^{4}Q^{2}k_{1}-\mathcal{B}%
Rk_{3}+\mathcal{B}Sk_{2}\right)  ^{2}+\left(  \mathcal{B}Pk_{3}-\mathcal{B}Sk_{1}%
+\mathcal{AW}^{2}k_{2}+\mathcal{W}^{4}Q^{2}k_{2}\right)
^{2}%
\end{align*}
hold.
\end{theorem}

\section{Curvature Ellipse of the B-DR soliton surface according to the
parallel transport frame field in $E^{4}$}

The curvature ellipse of a surface in 4-dimensional Euclidean space $E^{4}$
depends on the second fundamental form of the surface. The second fundamental
form of a surface in 3-dimensional space is defined as a symmetric bilinear
form, but the second fundamental form of a surface in 4-dimensional space is
expressed as a pair of symmetric bilinear forms, since there are two
independent normal vectors. If a surface $S$ is embedded in 4-dimensional
space, (depending on two different normal vectors) the second fundamental form
at each point is defined as%
\[
h_{N}(p):T_{p}S\times T_{p}S\longrightarrow%
\mathbb{R}
.
\]
Here, $T_{p}S$ is the tangent space at the point $p$ and there are two second
fundamental forms:

1. $h_{N_{1}}$: second fundamental form in the first normal direction,

2. $h_{N_{2}}$: second fundamental form in the second normal direction.

Using these two forms, the curvature ellipse is defined as follows:%
\[
\left\{  \left(  h_{N_{1}}(X,X\right)  ,h_{N_{2}}(X,X)):X\in T_{p}S,\text{
}\left\Vert X\right\Vert =1\right\}  .
\]

For more details about the curvature ellipse of surfaces, we refer to
\cite{Little}, \cite{Mochida}, \cite{Rouxel}, \cite{Wong}, and etc.

Now, let us recall the following invariants that characterize the curvature
ellipse of surfaces.

The determinant $\Delta(P)$ and matrix $A(P)$ for a surface $S\subset E^{4},$
given by a regular patch $S:\Psi(s,t),$ are defined with the aid of
(\ref{hij}) by%
\begin{equation}
\Delta(p)=\frac{1}{4}\det\left[
\begin{array}
[c]{cccc}%
h_{11}^{1} & 2h_{12}^{1} & h_{22}^{1} & 0\\
h_{11}^{2} & 2h_{12}^{2} & h_{22}^{2} & 0\\
0 & h_{11}^{1} & 2h_{12}^{1} & h_{22}^{1}\\
0 & h_{11}^{2} & 2h_{12}^{2} & h_{22}^{2}%
\end{array}
\right]  (p) \label{Delta}%
\end{equation}
and%
\begin{equation}
A(p)=\left[
\begin{array}
[c]{ccc}%
h_{11}^{1} & h_{12}^{1} & h_{22}^{1}\\
h_{11}^{2} & h_{12}^{2} & h_{22}^{2}%
\end{array}
\right]  (p), \label{A(p)}%
\end{equation}
respectively. With the aid of these invariants, one can give the following
classifications for the origin $p$ of the normal space $T_{p}^{\bot}S$:

\textbf{a)} If $\Delta(p)<0$, then the point $p$ lies outside the curvature
ellipse and such a point is called a hyperbolic point of $S$.

\textbf{b)} If $\Delta(p)>0$, then the point $p$ lies inside the curvature
ellipse and such a point is called an elliptic point of $S$.

\textbf{c)} If $\Delta(p)=0$, then the point $p$ lies on the curvature ellipse
and such a point is called a parabolic point of $S$. For this case, we have
the following detailed possibilities:

\ \ \textbf{i)} If $\Delta(p)=0$ and $K(p)>0$, then the point $p$ is an
inflection point of imaginary type.

\ \ \textbf{ii)} If $\Delta(p)=0$, $K(p)<0$ and $rank(A(p))=2$, then the
ellipse is non-degenerate; if $\Delta(p)=0$, $K(p)<0$ and $rank(A(p))=1$, then
the point $p$ is an inflection point of real type.

\ \ \textbf{iii)} If $\Delta(p)=0$ and $K(p)=0$, then the point $p$ is an
inflection point of flat type \cite{Mochida}.

Now, by using (\ref{hij}) in (\ref{Delta}) and (\ref{A(p)}), we obtain the
invariant $\Delta(p)$ as%
\begin{equation}
\Delta(p)=-\frac{\mathcal{B}^{2}+4Q^{2}\mathcal{AC}^{2}}{4\mathcal{W}^{6}}.
\label{delta1}%
\end{equation}

From (\ref{delta1}) and the above definitions, we have

\begin{theorem}
Let $S:\Psi=\Psi(s,t)$ be a solution of the B-DR equation according to the
ED$^{2}$-frame field in $E^{4}$. Then the origin $p$ of the normal space
$T_{p}^{\bot}S$ can be classified by the following cases:

\textbf{a)} If the inequality $\mathcal{B}^{2}+4Q^{2}\mathcal{AC}^{2}>0$ is
satisfied, then $p$ lies outside the curvature ellipse and so, it is a
hyperbolic point of $S$.

\textbf{b)} If the inequality $\mathcal{B}^{2}+4Q^{2}\mathcal{AC}^{2}<0$ is
satisfied, then $p$ lies inside the curvature ellipse and so, it is an
elliptic point of $S$.

\textbf{c)} If the conditions $\mathcal{B}^{2}+4Q^{2}\mathcal{AC}^{2}=0$ is
satisfied, then $p$ lies on the curvature ellipse and so, it is a parabolic
point of $S$. Also in this case; we have the following situations:

\ \ \textbf{i)} if "$\mathcal{A}>0$" and "$\mathcal{B}=$ $\mathcal{C}=0$",
then $p$ that is an inflection point of imaginary type;

\ \ \textbf{ii}) if "$\mathcal{A}<0$ and $\mathcal{B}^{2}+4Q^{2}%
\mathcal{AC}^{2}=0$" or "$\mathcal{A}=\mathcal{B}=0\neq\mathcal{C}$", then $p$
is non-degenerate;

\ \ \textbf{iii) }if $\mathcal{A}=\mathcal{B}=\mathcal{C}=0$ holds, then $p$
is an inflection point of flat type.
\end{theorem}

\begin{proof}
(a) and (b) are obvious from (\ref{delta1}). Now let us examine (c), i.e.
$\Delta(p)=0$.

If $\Delta(p)=0,$ then we have $\mathcal{B}^{2}+4Q^{2}\mathcal{AC}^{2}=0.$

i) If $K(p)>0,$then from (\ref{Gaussian}) it must be $\mathcal{A}>0$. Also, if
$\mathcal{B}=$ $\mathcal{C}=0$ then $\Delta(p)=0$.

ii) If $K(p)<0,$then from (\ref{Gaussian}) it must be "$\mathcal{A}<0$ and
$\mathcal{B}^{2}+4Q^{2}\mathcal{AC}^{2}=0$" or "$\mathcal{A}=\mathcal{B}=0$,
$\mathcal{C}\neq0$"$.$ From (\ref{hij}) and (\ref{A(p)}), we get $rank(A)=2$
and this proves (ii).

iii) If $K(p)=0,$then from (\ref{Gaussian}) it must be $\mathcal{A}%
=\mathcal{B}=\mathcal{C}=0$ and this completes the proof.
\end{proof}

\section{Application}

In this section, we construct a soliton surface $\Psi(s,t)$ associated with
the B-DR equation and find the parallel transport frame field of the
$s$-parameter curve $\Psi(s,t)$ for all $t$ in $E^{4}$. Additionally, we
obtain its geometric invariants $k$ and $h$, the Gaussian curvature $K$, the
mean curvature vector field $\vec{H}$ and Gaussian torsion $K_{N}$. To better
understand our example, we can visualize it by projecting the soliton surface
into 3-dimensional spaces.

Let us consider the soliton surface as%
\begin{equation}
\Psi(s,t)=\left(  \frac{\sin s+s}{2},\frac{\cos s}{\sqrt{2}},\frac{\sin
s-s}{2},-\frac{t}{2\sqrt{2}}\right)  . \label{exsurf}%
\end{equation}
Here, the $s$-parameter curves $\Psi(s,t)$ (for all $t$) of the soliton
surface (\ref{exsurf}) satisfies the B-DR equation (\ref{BDReq}) in $E^{4}$.

The parallel transport frame fields of the $s$-parameter curve $\Psi
=\Psi(s,t)$ for all $t$ on the B-DR soliton surface (\ref{exsurf}) are
obtained as%
\[
\mathcal{T}(s,t)=\left(  \cos^{2}\left(  \frac{s}{2}\right)  ,-\frac{\sin
s}{\sqrt{2}},\frac{\cos s-1}{2},0\right)  ,
\]%
\[
\mathfrak{\mathcal{P}}_{1}(s,t)=\left(  0,0,0,-1\right)  ,
\]%
\begin{align*}
&  \mathfrak{\mathcal{P}}_{2}(s,t)=\\
&  \left(
\begin{array}
[c]{l}%
-\frac{1}{2}\sin r\left(  \sqrt{2}\cos\left(  \frac{s}{\sqrt{2}}\right)  \sin
s+2\sin^{2}(\frac{s}{2})\sin\left(  \frac{s}{\sqrt{2}}\right)  \right)  +\cos
r\left(  \cos\left(  \frac{s}{\sqrt{2}}\right)  \sin^{2}(\frac{s}{2}%
)-\frac{\sin s\sin\left(  \frac{s}{\sqrt{2}}\right)  }{\sqrt{2}}\right)  ,\\
\frac{\cos\left(  r+\frac{s}{\sqrt{2}}\right)  \sin s}{\sqrt{2}}-\sin\left(
r+\frac{s}{\sqrt{2}}\right)  \cos s,\\
\frac{1}{2}\left(  -\cos\left(  \frac{s}{\sqrt{2}}\right)  \left(  \cos
r(1+\cos s)+\sqrt{2}\sin r\sin s\right)  +\sin\left(  \frac{s}{\sqrt{2}%
}\right)  \left(  2\cos^{2}(\frac{s}{2})\sin r-\sqrt{2}\cos r\sin s\right)
\right)  ,\\
0
\end{array}
\right)  ,
\end{align*}%
\begin{align*}
&  \mathfrak{\mathcal{P}}_{3}(s,t)=\\
&  \left(
\begin{array}
[c]{l}%
\frac{1}{2}\left(  \cos\left(  \frac{s}{\sqrt{2}}\right)  \left(  \left(  \cos
s-1\right)  \sin r-\sqrt{2}\cos r\sin s\right)  +\sin\left(  \frac{s}{\sqrt
{2}}\right)  \left(  \cos r\left(  \cos s-1\right)  +\sqrt{2}\sin r\sin
s\right)  \right)  ,\\
-\frac{1}{2}\cos\left(  \frac{s}{\sqrt{2}}\right)  \left(  2\cos r\cos
s+\sqrt{2}\sin r\sin s\right)  -\sin\left(  \frac{s}{\sqrt{2}}\right)  \left(
\frac{\sin s\cos r}{\sqrt{2}}-\sin r\cos s\right)  ,\\
-\frac{\sin s\cos\left(  r+\frac{s}{\sqrt{2}}\right)  }{\sqrt{2}}+\cos
^{2}(\frac{s}{2})\sin\left(  r+\frac{s}{\sqrt{2}}\right)  ,\\
0
\end{array}
\right)  ,
\end{align*}
where $r$ is a real constant. Also, the curvature functions $k_{1}$, $k_{2}$
and $k_{3}$ according to the parallel transport frame of (\ref{exsurf}) are
obtained by%
\[
k_{1}(s,t)=0,\text{ }k_{2}(s,t)=\frac{\sin\left(  r+\frac{s}{\sqrt{2}}\right)
}{\sqrt{2}},\text{ \ }k_{3}(s,t)=\frac{\cos\left(  r+\frac{s}{\sqrt{2}%
}\right)  }{\sqrt{2}}.
\]
On the other hand, we obtain the geometric invariants $k$, $h$ and the
Gaussian curvature, mean curvature vector field and Gaussian torsion of the
soliton surface (\ref{exsurf}) as%
\begin{align*}
&  h=0,\\
&  k=-4\left(  q_{1}\cos\left(  r+\frac{s}{\sqrt{2}}\right)  -q_{2}\sin\left(
r+\frac{s}{\sqrt{2}}\right)  \right)  ^{2},\\
&  K=2\left(  q_{2}\cos\left(  r+\frac{s}{\sqrt{2}}\right)  +q_{1}\sin\left(
r+\frac{s}{\sqrt{2}}\right)  \right)  ,\\
&  \vec{H}=\frac{1}{2\sqrt{2}}\left(  \left(  4q_{1}+\sin\left(  r+\frac
{s}{\sqrt{2}}\right)  \right)  \mathfrak{\mathcal{P}}_{2}+\left(  4q_{2}%
+\cos\left(  r+\frac{s}{\sqrt{2}}\right)  \right)  \mathfrak{\mathcal{P}}%
_{3}\right)  ,\\
&  K_{N}=0,
\end{align*}
where $q_{1}$ and $q_{2}$ are real constants.

Also, the determinant $\Delta(p)$ of the soliton surface (\ref{exsurf}) is%
\[
\Delta(p)=-\left(  q_{1}\cos\left(  r+\frac{s}{\sqrt{2}}\right)  -q_{2}%
\sin\left(  r+\frac{s}{\sqrt{2}}\right)  \right)  ^{2}%
\]
for all points $p$ and so, all points of the soliton surface are hyperbolic.

Finally, let us present the figures of the B-DR soliton surface (\ref{exsurf}%
)\ projections into $xyz,$ $xyw,$ $xzw$ and $yzw$-spaces. These projections
are shown in figures (a), (b), (c) and (d), respectively.

\begin{figure}[H]
\centering
\includegraphics[
height=1.8in, width=6.6in
]{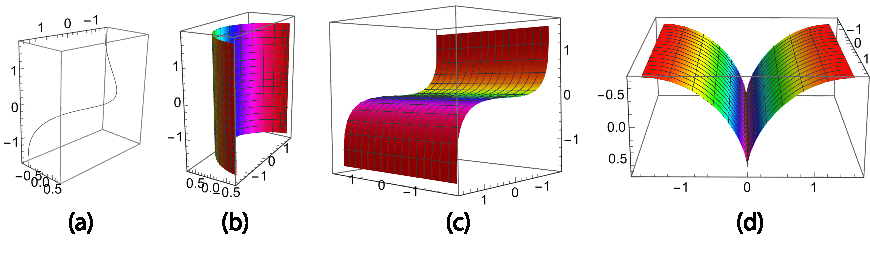}\caption{Projections of the B-DR soliton surface (\ref{exsurf}) }%
\label{fig:1}%
\end{figure}

\section{Conclusion and Future Work}

The geometric properties of a soliton surface $\Psi=\Psi(s,t)$ associated with
the Betchov-Da Rios (B-DR) equation using the parallel transport frame in
four-dimensional Euclidean space were examined in this study. First, for all
values of $t$, we obtained the derivative formulas for the parallel transport
frame field of a unit-speed $s$-parameter curve $\Psi=\Psi(s,t)$. By using
these frame formulations, we were able to calculate the soliton surface's
Gaussian curvature, mean curvature vector, Gaussian torsion and two basic
geometric invariants, $k$ and $h$. Following that, a set of theorems
describing the circumstances in which the soliton surface is flat, minimal or
semi-umbilic were established using these surface invariants.

Additionally, we established a significant theorem characterizing the
curvature ellipse of the B-DR soliton surface with respect to the parallel
transport frame field in $E^{4}$ by computing the determinant $\Delta
(\mathcal{P})$ and matrix $A(\mathcal{P})$ associated with the soliton
surface. Also, a theorem that describes Wintgen ideal (superconformal) B-DR
soliton surfaces based on the parallel transport frame field was proved by us,
offering important new information about these surfaces structures. In
conclusion, we construct an example of a B-DR soliton surface, identified its
geometric invariants, and showed its embedding into three-dimensional space to
demonstrate our theoretical understanding.

This study establishes a basis for future research in this area and provides a
fresh viewpoint on the geometric analysis of soliton surfaces resulting from
the B-DR equation. There are still a number of fascinating open problems that
need to be investigated further. Analyzing B-DR soliton surfaces in Minkowski
spacetime or four-dimensional Euclidean space with different frame fields is
one possible extension that could lead to new geometric classifications and
interpretations. Also, investigating soliton surfaces associated with the
visco-Da Rios equation is another exciting avenue that may yield further
geometrical and physical understanding, especially in the areas of fluid
dynamics and vortex filament theory. Future studies can strengthen the links
between differential geometry, soliton theory, and vortex dynamics by
following these avenues, which will improve our understanding of
higher-dimensional geometric structures and their practical uses.
\\
\\
\textbf{Conflict of interest} The authors declare that they have no
conflict of interest, regarding the publication of this paper.

\end{document}